\documentclass[12pt]{amsart}

\usepackage{amssymb}
\usepackage{url}
\usepackage{amscd}
\usepackage{geometry}
\usepackage{enumerate}
\usepackage{color}

\usepackage{graphicx}

\theoremstyle{plain}
\newtheorem{theorem}{Theorem}[section]
\newtheorem{corollary}[theorem]{Corollary}
\newtheorem{lemma}[theorem]{Lemma}
\newtheorem{proposition}[theorem]{Proposition}

\theoremstyle{remark}
\newtheorem{remark}{Remark}[section]

\usepackage{color}

\begin{document}


\title[Large spectral gaps for Steklov eigenvalues]{Large spectral gaps for Steklov eigenvalues under volume constraints and under localized conformal deformations.}

\author{Donato Cianci and Alexandre Girouard}

\begin{abstract}
In this paper we construct compact manifolds with fixed boundary geometry which admit Riemannian metrics of unit volume with arbitrarily large Steklov spectral gap. We also study the effect of localized conformal deformations that fix the boundary geometry. For instance, we prove that it is possible to make the spectral gap arbitrarily large using conformal deformations which are localized on domains of small measure, as long as the support of the deformations contains and connects each component of the boundary.
\end{abstract}

\address{University of Michigan, Department of Mathematics, 530 Church Street, Ann Arbor, MI 48109, USA}

\email{cianci@umich.edu}

\address{D\'{e}partement de Math\'{e}matiques et de Statistique, Pavillon Alexandre-Vachon, Universit\'{e} Laval, Qu\'{e}bec, QC, G1V 0A6, Canada}

\email{alexandre.girouard@mat.ulaval.ca}


\maketitle


\newcommand\R{\mathbb{R}}
\newcommand\N{\mathbb{N}}

\newcommand\cont{\operatorname{cont}}
\newcommand\diff{\operatorname{diff}}

\newcommand{\dvol}{\text{dvol}}

\newcommand{\GL}{\operatorname{GL}}
\newcommand{\myO}{\operatorname{O}}
\newcommand{\myP}{\operatorname{P}}
\newcommand{\eye}{\operatorname{Id}}
\newcommand{\myF}{\operatorname{F}}
\newcommand{\vol}{\operatorname{vol}}
\newcommand{\odd}{\operatorname{odd}}
\newcommand{\even}{\operatorname{even}}
\newcommand{\ol}{\overline}
\newcommand{\mye}{\operatorname{E}}
\newcommand{\myo}{\operatorname{o}}
\newcommand{\myt}{\operatorname{t}}
\newcommand{\irr}{\operatorname{Irr}}
\newcommand{\mydiv}{\operatorname{div}}
\newcommand{\re}{\operatorname{Re}}
\newcommand{\im}{\operatorname{Im}}
\newcommand{\supp}{\operatorname{supp}}
\newcommand{\scal}{\operatorname{scal}}
\newcommand{\spec}{\operatorname{spec}}
\newcommand{\tr}{\operatorname{trace}}
\newcommand{\sgn}{\operatorname{sgn}}
\newcommand{\SL}{\operatorname{SL}}
\newcommand{\myspan}{\operatorname{span}}
\newcommand{\mydet}{\operatorname{det}}
\newcommand{\SO}{\operatorname{SO}}
\newcommand{\SU}{\operatorname{SU}}
\newcommand{\specl}{\operatorname{spec_{\mathcal{L}}}}
\newcommand{\fix}{\operatorname{Fix}}
\newcommand{\id}{\operatorname{id}}
\newcommand{\singsup}{\operatorname{singsupp}}
\newcommand{\wave}{\operatorname{wave}}
\newcommand{\ind}{\operatorname{ind}}
\newcommand{\lie}{\operatorname{Lie}}
\newcommand{\mynull}{\operatorname{null}}
\newcommand{\floor}[1]{\left \lfloor #1  \right \rfloor}

\newcommand\restr[2]{{
  \left.\kern-\nulldelimiterspace 
  #1 
  \vphantom{\big|} 
  \right|_{#2} 
  }}

\section{Introduction}
In this paper we construct compact manifolds with fixed boundary geometry which admit Riemannian metrics of unit volume with arbitrarily large Steklov spectral gap. 
Recall that the
\emph{Steklov problem} on a Riemannian manifold $M$ with boundary $\Sigma$ is:
\begin{equation*}
\begin{cases}
  \Delta u=0& \mbox{ in } M,\\
  \partial_\nu u=\sigma \,u& \mbox{ on }\Sigma,
\end{cases}
\end{equation*}
where $\Delta=\mydiv \circ \nabla$ is the Laplace-Beltrami operator acting on the
space $C^\infty(M)$ of smooth functions on
$M$, and $\partial_\nu$ is the outward normal derivative.
The real numbers $\sigma$ for which a nonzero solution $u$ to the
above partial differential equation exists are called the {\it Steklov
  eigenvalues} of $M$. They form a sequence  
$0=\sigma_1<\sigma_2\leq\sigma_3\leq\cdots\nearrow\infty$, each
repeated according to its multiplicity, which is called the {\it Steklov
  spectrum} of $M$. There is a large literature on isoperimetric type problems for the Steklov eigenvalues. Upper bounds have been obtained under various constraints, for instance the measure of the boundary for a simply connected planar domain \cite{wein}, the genus and number of boundary components of a surface \cite{MR3190427,fraschoen,gp3}, the isoperimetric ratio \cite{MR2807105}, and more recently the diameter of Euclidean domains \cite{MR3640624}. See also \cite{MR3662010} and the references therein. 

  In \cite{girouard_preprint}, families of compact manifolds for which the Steklov spectral gap $\sigma_2$ grows arbitrarily large despite having fixed boundary were produced. 
 However, the volume of the manifolds in these families also becomes arbitrarily large. Our main result presents a family of metrics with fixed boundary geometry and fixed volume for which the Steklov spectral gap grows arbitrarily large. 
\begin{theorem}\label{thm:fixed_vol}
 Let $M$ be a smooth, compact, connected manifold of dimension at least four with connected boundary $\Sigma$. Suppose there is a Riemannian metric $g_\Sigma$ on $\Sigma$ which admits a unit Killing vector field $\xi\in\Gamma(T\Sigma)$ with dual 1-form $\eta$ whose exterior derivative is nowhere zero. Then there exists a family $\{g_\epsilon\}_{\epsilon>0}$ of Riemannian metrics on $M$ which coincide with $g_\Sigma$ on $\Sigma$ such that $\vol(g_\epsilon)=1$ and $\sigma_2(M,g_\epsilon)\to\infty$ as $\epsilon\searrow 0$.
\end{theorem}  

Note that it is also possible to construct the family $g_\epsilon$ so that $\sigma_2(M,g_\epsilon)\to\infty$, $\vol(M,g_\epsilon)\to 0$, and  diam$(M,g_\epsilon)\to 0$ as $\epsilon\searrow 0$. See Remark \ref{rem:smallvolumediameter}.
 We do not know if a statement analogous to Theorem \ref{thm:fixed_vol} holds in dimension three. Note that on surfaces $\sigma_2$ is bounded above in terms of the area and genus \cite{MR3190427,MR2770439}. The proof of Theorem \ref{thm:fixed_vol} will be presented in Section \ref{section:gaps_with_fixed_volume}. It is based on a construction due to Bleecker of closed manifolds with large Laplace-spectral gap and fixed volume form \cite{bleecker}. It applies for instance to the situation where $\Sigma$ is isometric to a round sphere of odd dimension.

If we remove the volume constraint, then it is easier to construct Riemannian metrics that have arbitrarily large Steklov eigenvalues while keeping the boundary fixed. 
Indeed, it is enough to consider conformal deformations $e^hg$ where the smooth function $h$ belongs to the space
$$C^\infty_0(M)=\{h\in C^\infty(M)\,:\,h=0\mbox{ on }\Sigma\}.$$
For any manifold with dimension at least 3, Colbois, El Soufi and the second author \cite{girouard_preprint} constructed a family of smooth functions $h_\epsilon\in C^\infty_0(M)$ such that
$\sigma_2(M,e^{h_\epsilon}g)\to\infty$ as $\epsilon\searrow 0$.
This family has the additional property that for any point $x$ in the interior of $M$, 
$\lim_{\epsilon\to 0}h_\epsilon(x)=+\infty.$
The following theorem is a refinement of this construction. It shows that it suffices to localize the conformal deformations on an arbitrarily small connected neighborhood of the boundary. 
\begin{figure}\label{figure:localized}
  \includegraphics[width=10cm]{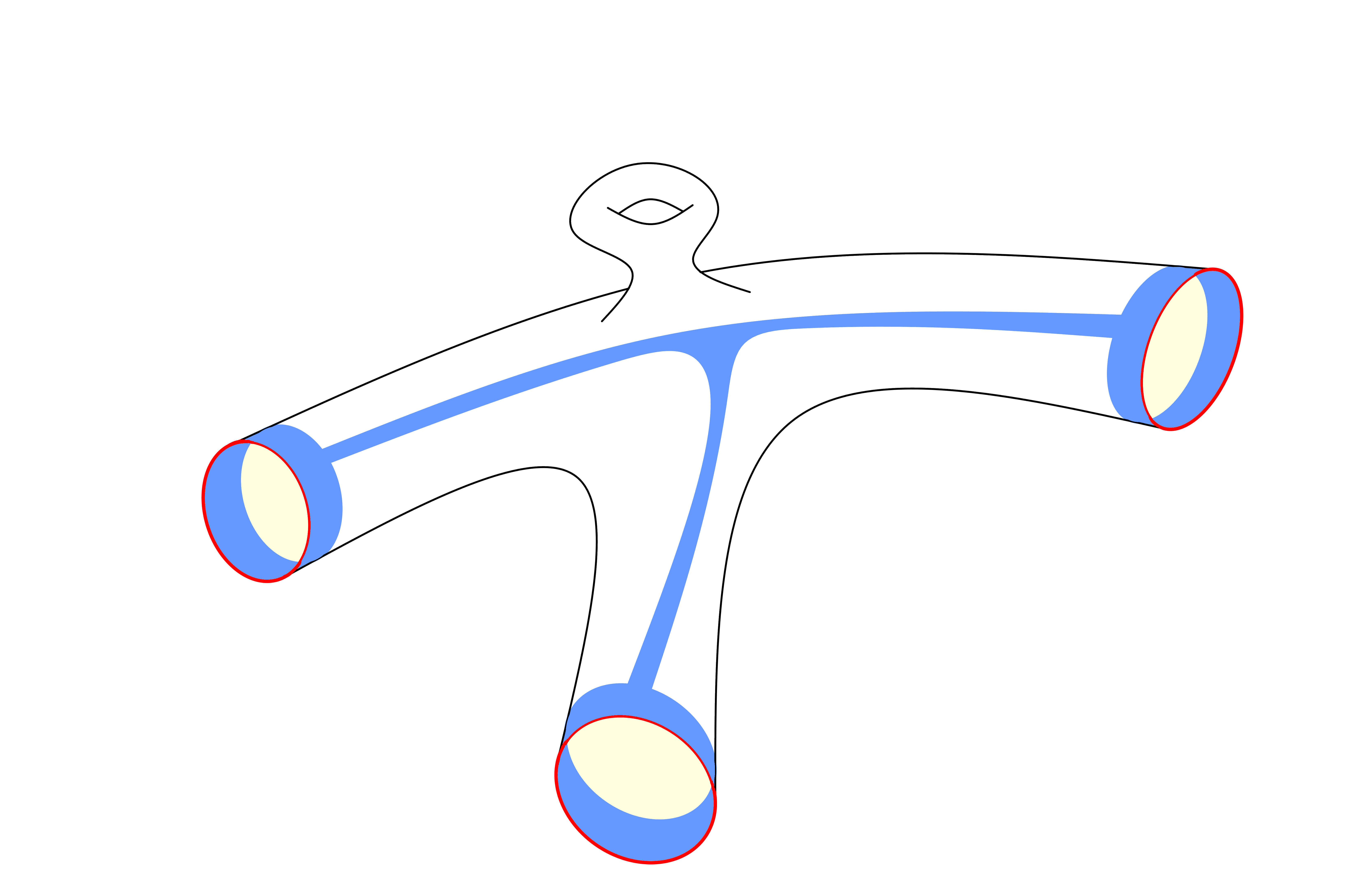}
  \caption{On this manifold it is possible to
    make $\sigma_2$ arbitrarily large by using a conformal deformation that
    is localized in the shaded connected neighborhood of the boundary.}
\end{figure}
\begin{theorem}
 \label{main_result1}
  Let $(M,g)$ be a smooth connected compact manifold of dimension
  $\geq 3$  with boundary $\Sigma$. Let $\Omega\subset M$ be a connected
  neighborhood of the boundary $\Sigma$, then there exists a family
  $h_\epsilon\in C^\infty_0(\Omega)$, with $\epsilon>0$, such that
  $$\lim_{\epsilon\to 0}\sigma_2(M,e^{h_\epsilon}g)=+\infty.$$
\end{theorem}

Theorem \ref{main_result1} is a special case of the more general Theorem \ref{thm:localizedgap} below, with $b=1$. Applying a diagonal argument, immediately leads to the following corollary (see Figure \ref{figure:localized}):
\begin{corollary}
  There exists a family $h_\epsilon\in C^\infty_0(M)$ such that
  $\vol_g(\supp h_\epsilon)\to 0$ 
    and $\sigma_2(M,e^{h_\epsilon}g)\to+\infty$ as $\epsilon\to 0$.
\end{corollary}

In order to understand what happens when the domain $\Omega$ is not connected, we introduce a family of
localized conformal invariants. Given a neighborhood $\Omega\subset
M$ of the boundary $\Sigma$, define:
$$\sigma^{\#}_j(M,\Omega)=\sup\{\sigma_j(M,e^hg)\,:\,h\in C^\infty_0(M),\,
\supp(h)\subset\Omega\}.$$
The next result shows that the finiteness of $\sigma^{\#}_j(M,\Omega)$ depends
on the number of connected components of $\Omega$ that have nontrivial
intersection with the boundary.   
\begin{theorem}
  \label{thm:localizedgap}
  Let $(M,g)$ be a smooth connected compact manifold of dimension at
  least three with boundary $\Sigma$. Let $\Omega\subset M$ be a
  neighborhood of the boundary $\Sigma$.
  If $b\geq 1$ is the number of connected components of $\Omega$ 
  which intersect the boundary $\Sigma$, 
  then $\sigma_{b}^\#(M,\Omega)<\infty$ and $\sigma_{b+1}^\#(M,\Omega)=\infty$. 
\end{theorem}

 If $\Omega$
is a thin collar 
neighborhood of $\Sigma$, then $b$ is the number of connected
components of the boundary $\Sigma$. One should compare Theorem \ref{thm:localizedgap} with
\cite[Theorem 1.1]{girouard_preprint}.

\begin{remark}
In the present paper the geometry of the boundary is fixed in the sense that the Riemannian metric $g_\Sigma$ on the boundary $\Sigma$ is given. One can also study the situation where $\Sigma\subset\R^m$ is a prescribed $n$-dimensional submanifold and $M\subset\R^m$ is allowed to vary  among all $(n+1)$-dimensional submanifolds with boundary $\Sigma$. In this context Colbois, Gittins and the second author [CGG] proved that $\sigma_2$ is bounded above in terms of the volume of $M$. More precisely, they proved the existence of a constant $A=A(\Sigma)$ which depends on the embedding $\Sigma\subset\R^m$ such that
$\sigma_2\leq A\vol(M)$. If the Riemannian manifolds $(M,g_\epsilon)$ constructed in Theorem \ref{thm:fixed_vol} are isometrically embedded in some Euclidean space $\R^m$ (say using the Nash embedding theorem) then it follows that the corresponding constants $A(\Sigma_\epsilon)\to \infty$ as $\epsilon\searrow 0$. In particular, the distortion of $\Sigma_\epsilon$ tends to $+\infty$ (see [CGG] for details regarding the distortion).
\end{remark}

\subsection*{Acknowledgments}
DC was supported by a CRM-Laval Postdoctoral Fellowship during the year 2016/2017 when this project was started. AG acknowledges the support of NSERC.

\section{Large Steklov eigenvalues on manifolds with fixed volume and 
  boundary}
  \label{section:gaps_with_fixed_volume}

The goal of this section is to prove Theorem \ref{thm:fixed_vol}. 
Recall the following result, due to Bleecker:
\begin{theorem}[\cite{bleecker}, Theorem 2.3]
\label{bleecker_theorem}
Let $(\Sigma,g_\Sigma)$ be a closed, connected, Riemannian manifold. Let $\xi$ be a unit Killing vector field on $\Sigma$ with dual 1-form $\eta$ such that $d\eta$ is nowhere zero. Let $\lambda_2(t)$ be the first nonzero eigenvalue of the Laplacian for the metric 
$$g_\Sigma(t)=t^{-1}g_\Sigma +(t^n - t^{-1}) \eta \otimes \eta.$$
Then there exists $\delta>0$ such that
$\lambda_2(t)\geq\delta t,$ for $t\geq 1$.
\end{theorem} 

A crucial feature of the metrics $g_\Sigma(t)$ is the fact that they induce the same volume form $d\Omega$ for each $t\geq 1$. This allows us to pick an orthonormal basis $\{\pi_i\}_{i=1}^\infty$ of $L^2(\Sigma, d\Omega)$ which is independent of $t$. This basis can be used to perform a Fourier decomposition of a test function for the variational characterization of $\sigma_2(M,g_{\epsilon})$ on $M$ close to its boundary $\Sigma$ and estimate its Dirichlet energy (see \eqref{fourier_decomp} and \eqref{dirichlet_estimate}).  
\begin{proof}[Proof of Theorem \ref{thm:fixed_vol}]
In our situation, where $\Sigma$ is the boundary of $M$, we identify a neighborhood of $\Sigma$ in $M$ with $\Sigma\times[0,3).$
The above family of Riemannian metrics $g_\Sigma(t)$ will be used to extend $g_\Sigma$ to the interior of $M$ using a cutoff function.
Given $\epsilon>0$, let $f_{\epsilon} \colon [0,3) \to \mathbb{R}_{\geq 1}$ be a smooth function satisfying:
\begin{align*}
f_{\epsilon}(t) = \begin{cases}
1+ \epsilon^{-3}t \hspace{5 pt} &\text{for } t \in [0,1) \\
1 \hspace{5 pt} &\text{for } t \in [2, 3).
\end{cases}
\end{align*}  
On the cylindrical part of $M$ we use coordinates $(x,t)\in \Sigma \times [0,3)$ to define the following metric:
\begin{gather}\label{eq:themetric}
g_{\epsilon} = g_\Sigma(f_{\epsilon}(t))+dt^2.
\end{gather}
Then, after extending $g_\epsilon$ outside of the cylindrical part, we obtain a smooth Riemannian metric on $M$, which we also denote by $g_\epsilon$.  The restriction of $g_\epsilon$ to the boundary is $g_\Sigma$.
 By choosing the same extension for every $\epsilon>0$ we ensure that $\vol(M,g_\epsilon)$ is independent of $\epsilon$. We will show that $\sigma_2(M,g_\epsilon)$ diverges as $\epsilon \searrow 0$.

Observe that the volume form on $\Sigma$ induced from the metric $g_\Sigma(t)$ is independent of $t$, hence coincides with the volume form induced from $g_\Sigma=g_\Sigma(1)$, which we denote by $d\Omega$. Moreover, let $\{\pi_i\}_{i=1}^{\infty}$ be an orthonormal basis for $L^2(\Sigma, d\Omega)$ consisting of eigenfunctions for the Laplace-Beltrami operator induced from the metric $g_\Sigma$. 

The idea is to locally express a test function for the variational characterization of $\sigma_2$ in a Fourier series near the boundary. The Rayleigh quotient for the variational characterization of Steklov eigenvalues is:
\begin{align*}
R(u,g_\epsilon)  =  \frac{\int_M |d u|^2 \dvol(g_\epsilon)}{\int_{\Sigma} |u|^2 d\Omega}.
\end{align*}
Let $u$ be a smooth function on $M$ such that $\int_{\Sigma} u \hspace{3 pt} d\Omega = 0$ and $\int_{\Sigma} |u|^2 \hspace{3 pt} d\Omega = 1$. If we restrict $u$ to $\Sigma \times [0,1)$, then we have the following Fourier series representation for $u(x,t)$:
\begin{align*}
u(x,t) = \sum_{i=1}^{\infty} a_i(t) \pi_i(x).
\end{align*} 
Our assumptions on $u$ translate into the following conditions on the Fourier coefficients:
\begin{align*}
a_1(0) = 0 \hspace{4 pt} \text{ and } \sum_{i = 2}^{\infty} a_i(0)^2 = 1.
\end{align*} 
Using the fact that the functions $\pi_i$ do not depend on $t$, the exterior derivative of $u$ is simply given by:
\begin{align}
\label{fourier_decomp}
du(x,t) = \sum_{i = 1}^{\infty} a_i'(t) \pi_i(x) dt+ a_i(t) d \pi_i(x).
\end{align}
The rest of the proof is similar to that of Proposition 3.1 in \cite{girouard_preprint}.
The following estimate on $R(u, g_\epsilon)$ follows from \eqref{fourier_decomp}:
\begin{align}
\label{dirichlet_estimate}
R(u,g_{\epsilon}) &= \int_M |du|^2 \dvol(g_\epsilon)\\
&\geq \int_{\Sigma \times (0,1)} |du|^2 \dvol(g_{\epsilon}) \nonumber\\
&= \sum_{i = 1}^{\infty} \int_0^1\int_{\Sigma} 
\left(|a_i'(t)|^2 |\pi_i(x)|^2 
 +|a_i(t)|^2 |d \pi_i(x)|_{g_\Sigma(f_\epsilon(t))}^2\right) \hspace{3 pt} d\Omega(x) dt.\nonumber
\end{align}
It follows from Theorem \ref{bleecker_theorem} that $\lambda_2(g_\Sigma(f_\epsilon(t)))\geq \delta f_\epsilon(t)$. Using the variational characterization of Laplace eigenvalues, we have:
\begin{align}
\label{eig_bound}
\int_{\Sigma} | d \pi_i|^2_{g_\Sigma(f_{\epsilon(t)})} \hspace{3 pt} d\Omega(x) \geq
\lambda_2(g_\Sigma(f_\epsilon(t)))\int_{\Sigma} \pi_i^2 \hspace{3 pt} d\Omega(x)
\geq
 \delta \left( 1+ \epsilon^{-3}t \right).
\end{align} 
Applying \eqref{eig_bound}, we obtain the following lower bound on the Rayleigh quotient:
\begin{align*}
R(u,g_{\epsilon}) \geq \sum_{i=1}^\infty \int_0^1 \bigl(|a'_i(t)|^2 + \delta |a_i(t)|^2\left( 1+ \epsilon^{-3}t \right)\bigr) dt
\end{align*}
We have the following two cases to consider: \\
\noindent {\it Case 1.}  There exists $t_0 \in (0, \epsilon)$ such that $|a_i(t_0)| \leq \frac{1}{2} |a_i(0)|$, which implies $|a_i(0)-a_i(t_0)|\geq |a_i(0)| - |a_i(t_0)| \geq \frac{1}{2}|a_i(0)|$. Thus, by the Cauchy-Schwarz inequality, we have:
\begin{align*}
\int_0^1 |a_i'(t)|^2 dt \geq \int_0^{t_0} |a_i'(t)|^2 dt \geq \frac{1}{t_0} \left( \int_0^{t_0} a_i'(t) dt \right)^2 &\geq \frac{1}{\epsilon} \left(a_i(0) - a_i(t_0) \right)^2\\
&\geq \frac{1}{4 \epsilon} |a_i(0)|^2. 
\end{align*}
\noindent {\it Case 2.} For every $t \in (0, \epsilon)$, $|a_i(t)| \geq \frac{1}{2} |a_i(0)|$. Then:
\begin{align*}
\int_0^1 \delta |a_i(t)|^2\left( 1+ \epsilon^{-3}t \right) dt \geq  \int_0^{\epsilon} \delta |a_i(t)|^2\left( 1+ \epsilon^{-3}t \right) dt &\geq \frac{\delta}{4} |a_i(0)|^2\int_0^{\epsilon} \left( 1+ \epsilon^{-3}t \right) dt \\
&\geq \frac{\delta}{8 \epsilon}|a_i(0)|^2.
\end{align*} 
Thus, in either case, for $A = \min\{ 1/4, \delta/8 \}$ we have:
\begin{align*}
R(u,g_{\epsilon}) &\geq \sum_{i=1}^\infty \int_0^1 |a'_i(t)|^2 + \frac{1}{4}|a_i(t)|^2\left( 1+ \epsilon^{-3}t \right) dt \\
&\geq \frac{A}{\epsilon} \sum_{i=2}^\infty |a_i(0)|^2 \\
&= \frac{A}{ \epsilon}. 
\end{align*}
By the variation characterization of $\sigma_2(g_\epsilon)$, the lower bound follows. 
\end{proof}

\begin{remark}\label{rem:smallvolumediameter}
If one replaces the definition \eqref{eq:themetric} of the metric $g_\epsilon$ with
\begin{gather*}\label{eq:themetric}
g_{\epsilon} = h_\epsilon(t)(g_\Sigma(f_{\epsilon}(t))+dt^2),
\end{gather*}
where $h_\epsilon\in C^\infty(M)$ is an appropriate positive function which is small away from the boundary $\Sigma$, then one can also ensure that 
$\vol(M,g_\epsilon)\to 0$ and diam$(M,g_\epsilon)\to 0$ as $\epsilon\searrow 0$.
\end{remark}

\section{Localized conformal deformations and the Steklov spectrum}
  \label{gaps_with_fixed_boundary}
The goal of this section is to prove Theorem \ref{thm:localizedgap}.
We will make use of the mixed Steklov-Neumann problem, which we now recall. Partition the connected components of $\Sigma$ into two sets labeled $\Sigma^S \neq \emptyset$ and $\Sigma^N$. The mixed Steklov-Neumann problem is the following: 
 \begin{align}
 \label{mixed_prob}
 \begin{cases}
 \Delta u = 0 \hspace{4 pt} \text{in } M,\\
 \partial_{\nu} u = \sigma u \hspace{4 pt} \text{on }\Sigma^S,\\
 \partial_{\nu} u = 0 \hspace{4 pt} \text{on } \Sigma^N.
 \end{cases}
 \end{align}
Its spectrum is denoted as follows:
 \begin{equation*}
 0 = \sigma_1^N < \sigma_2^N \leq \sigma_3^N \leq \cdots\nearrow \infty. 
 \end{equation*} 
The Steklov-Neumann eigenvalues have the following min-max variational characterization:
\begin{align*}
\sigma_k^N(g)  = \inf_{E \in \mathcal{H}_k} \sup_{u \neq 0 \in E} \frac{\int_M |d u|^2 \dvol(g)}{\int_{\Sigma^S} |u|^2 \dvol(g_\Sigma)},
\end{align*}
and similarly the max-min variational characterization:
\begin{align*}
\sigma_k^N(g) = \sup_{E \in \mathcal{H}_{k-1}} \inf_{u \neq 0 \in E^{\perp}} \frac{\int_M |d u|^2 \dvol(g)}{\int_{\Sigma^S} |u|^2 \dvol(g_\Sigma)},
\end{align*}
where $\mathcal{H}_k$ is the set of $k$-dimensional subspaces of $C^\infty(M)$. We start by observing the following lemma, which adapts Proposition 3.3 in \cite{girouard_preprint} to the mixed Steklov-Neumann setting. 

\begin{lemma}
\label{mixed}
Let $(M,g)$ be a compact, connected, smooth Riemannian manifold with boundary $\Sigma$ and dimension $n+1 \geq 3$. As above, partition the connected components of $\Sigma$ into two sets: $\Sigma^S \neq \emptyset$ and $\Sigma^N$. Assume that there is a neighborhood of $\Sigma^S$ (in $M$) that is isometric to $\Sigma^S \times [0,L)$ for some $L>0$. Then there exists $C>0$ such that for every $\epsilon>0$ sufficiently small, there exists a metric $g_\epsilon = e^{\delta_{\epsilon}} g$ which coincides with $g$ in the neighborhood $\Sigma^S \times [0, \epsilon)$ of $\Sigma$ such that:
\begin{align*}
\sigma^N_2(g_\epsilon) \geq \frac{C}{\epsilon}.
\end{align*}
\end{lemma}

\begin{proof}
Let $\delta_{\epsilon} \geq 0$ be a smooth function that is identically equal to $-2\log(\epsilon)$ in $M \backslash(\Sigma^S \times [0,L/2])$ and identically zero on $\Sigma^S \times [0,\epsilon)$ (assume $\epsilon<1$). Set $g_{\epsilon} = e^{2 \delta_\epsilon}g$. The strategy is to locally decompose a test function $u$ on $M$ into a Fourier series using the Laplace eigenfunctions of $\Sigma^S$. Then we will use the Fourier coefficients to obtain a lower bound for the Dirichlet energy of $u$.  

Let $\Sigma_1^S,..., \Sigma_b^S$ be a decomposition of $\Sigma^S$ into connected components. Let $\{ \phi_j \}_{j=1}^{\infty}$ be an orthonormal basis for $L^2(\Sigma^S)$ given by eigenfunctions of the Laplace-Beltrami operator on $\Sigma^S$ with corresponding eigenvalues $\{ \lambda_j \}_{j=1}^{\infty}$. Let $u$ be a smooth function on $M$ that is orthogonal to constant functions on $\Sigma^S$ and is normalized so that $\int_{\Sigma^S}|u|^2 \dvol(g_\Sigma)=1$. When we restrict $u$ to $\Sigma^S \times [0,L)$, we get:
\begin{align*}
u(x,t)  = \sum_{j \geq 1} a_j(t) \phi_j(x),
\end{align*}  
where:
\begin{align*}
a_j(0) = \left| \Sigma^S_j \right|^{-1/2} \int_{\Sigma^S_j} u \hspace{3 pt} \dvol(g_\Sigma), \hspace{4 pt} \text{for }j = 1,...,b.
\end{align*}
The constraints on $u$ translate into the following:
\begin{align*}
a_1(0) \left| \Sigma^S_1 \right|^{1/2} + ... + a_b(0) \left| \Sigma^S_b \right|^{1/2} = 0,  
\end{align*}
and
\begin{align*}
\sum_{j \geq 1} a_j(0)^2 = 1. 
\end{align*}
Notice also that:
\begin{align}
\label{df-expansion}
du(x,t) = \sum_{j \geq 1} \left( a'_j(t) \phi_j(x) dt + a_j(t) d \phi_j(x) \right). 
\end{align}
Substituting this expression into the Rayleigh quotient for $u$ yields:
\begin{align}
R(u, g_{\epsilon}) &= \int_M |d u|^2 \dvol(g_\epsilon) \nonumber \\
&  \geq \int_{\Sigma^S \times (0,L)} |du|^2 e^{(n-1)\delta_\epsilon}\dvol(g) \nonumber \\
& \geq \sum_{j \geq 1} \int_0^L \left( a'_j(t)^2 + \lambda_j a_j(t)^2 \right) e^{(n-1)\delta_\epsilon} dt.
\label{contrast} 
\end{align} 
From \eqref{contrast}, we see two ways to make $R(u,g_\epsilon)$ large; either the function $a_j$ decreases quickly as we move away from $\Sigma^S$, making the first term large in \eqref{contrast}, or $a_j$ remains large and the second term contributes to the Dirichlet energy. The argument in Proposition 3.1 of \cite{girouard_preprint} with $\Sigma$ replaced by $\Sigma^S$ yields a rigorous version of this heuristic argument. Indeed, following the proof of Proposition 3.1 in \cite{girouard_preprint} one gets:
\begin{align}
\label{lower_bound}
R(u,g_\epsilon) \geq \frac{A}{\epsilon} \left( 1 - \sum_{j \leq b} a_j(0)^2 \right),
\end{align}
where $A = \frac{1}{4}\min \left \{ \lambda_{b+1}, \frac{1}{4} \right \}$. Notice that if $u$ is assumed to be orthogonal to constant functions on $\Sigma^S$, then $\sum_{j \leq b} a_j(0)^2 = 0$. It follows that $\sigma^N_{b+1}(g_\epsilon) \geq \frac{A}{\epsilon}$. 
Thus, it suffices to assume that $\sum_{j \leq b} a_j(0)^2 \geq \frac{1}{2}$ and $b \geq 2$. 

Let $\Omega$ be the complement in $M$ of $\Sigma^S \times [0, L/2)$. Following the argument in Proposition 3.3 of \cite{girouard_preprint} with $\Sigma_j$ replaced by $\Sigma^S_j$, one gets the following lower bound for  $R(u, g_\epsilon)$ in this case:
\begin{align*}
R(u, g_\epsilon) \geq \frac{B}{2 \epsilon},
\end{align*}
where:
\begin{align*}
B = \frac{\min \{ \mu (\Omega, g) b L, \frac{1}{2b} \}}{32 (b-1)^2}\frac{\min_{j \leq b} |\Sigma^S_j |}{\max_{j \leq b} |\Sigma_j^S |},
\end{align*}
and $\mu(\Omega, g)$ is the first positive Neumann eigenvalue of the Laplacian in $\Omega$. 
Thus, by the min-max characterization of the mixed problem, we conclude that $\sigma^N_2(g_\epsilon) \geq \frac{C}{\epsilon}$, which completes the proof of the lemma. 
\end{proof}

We turn our attention towards studying $\sigma^{\#}_i(\Omega)$, where $\Omega$ is a domain in $M$. Let $\partial_E \Omega := \partial \Omega \cap \Sigma$ denote the {\it exterior boundary} of $\Omega$ and $\partial_I \Omega := \partial \Omega \backslash \left( \partial_E \Omega \right)$ denote the {\it interior boundary}. Throughout, we will assume that the exterior boundary of $\Omega$ coincides with $\Sigma$ (i.e. that $\Omega$ contains the boundary). We leverage Lemma \ref{mixed} to study $\sigma_i^{\#}(\Omega)$ and show that whether it is finite depends on the number of connected components of $\Omega$ that intersect the boundary of $M$. The idea is, roughly, to use the Steklov-Neumann problem to uncouple the different connected components of $\Omega$ from the rest of $M$. Then we will compare the Steklov-Neumann eigenvalues of the uncoupled domains to the Steklov eigenvalues of the conformal metric on $M$. The result will follow from an application of Lemma \ref{mixed}. The next proposition shows that if there are $b$ connected components of $\Omega$ for which $\partial_E\Omega \neq \emptyset$, then $\sigma^{\#}_b(\Omega)$ is finite:

\begin{proposition}
\label{half_theorem}
Let $(M,g)$ be a compact, connected, smooth manifold of dimension $n+1 \geq 3$ with boundary $\Sigma$ and suppose that $\Omega \subset M$ is a closed subdomain with smooth boundary and $\partial_E \Omega = \Sigma$. Moreover, suppose that $\Omega$ has $b$ connected components for which $\partial_E\Omega \neq \emptyset$. If $e^{\delta} g$ is a conformal deformation of $g$ with $\supp(\delta) \subset \Omega^o$, then:
\begin{align*}
\sigma_{b} (e^{\delta}g) \leq C,
\end{align*}
where $C>0$ is a constant that depends on $g$ and $\Omega$. Consequently, $\sigma^{\#}_{b}(\Omega)<\infty$. 
\end{proposition}

\begin{proof}
We construct test functions and use the variational characterization of the Steklov eigenvalues. First, some notation. Enumerate the connected components of $\Omega$ for which $\partial_E \Omega \neq \emptyset$ as $\{\Omega_i\}_{1\leq i \leq b}$. Moreover, set $\Sigma_i = \partial_E \Omega_i$. For each $i$ let $\widetilde{\Omega}_i$ be an open set containing $\Omega_i$ for which $\widetilde{\Omega}_i \cap \widetilde{\Omega}_j = \emptyset$. Then for each $i$, let $\psi_i(x)$ be a smooth function satisfying:
\begin{align*}
\psi_i(x) = \begin{cases}
|\Sigma_i|^{-1/2} \hspace{5 pt} &\text{for } x \in \Omega_i\\
0 \hspace{5 pt} &\text{for } x \in \widetilde{\Omega}_i^c. 
\end{cases}
\end{align*}
Observe that:
\begin{align*}
\int_M |d \psi_j |^2 \dvol(e^\delta g) = \int_{\widetilde{\Omega}_i \backslash \Omega_i} |d \psi_j|^2 \dvol(g) < C,
\end{align*}
and $\int_\Sigma |\psi_i|^2 \dvol(g_\Sigma) = 1$. Finally, notice that the $\psi_i$ are mutually orthogonal on the boundary. Thus, by the min-max characterization of the eigenvalues, we have: $\sigma_{b}(e^{\delta}g) <C$, as required. 
\end{proof}

We pivot now to higher Steklov eigenvalues. The next proposition demonstrates that there are conformal deformations which make $\sigma_{b+1}$ arbitrarily large, assuming that a neighborhood of the boundary is isometric to a Riemannian product. 

\begin{proposition}
\label{auxiliary}
Let $(M,g)$ be a compact, connected, smooth manifold of dimension $n+1 \geq 3$ with boundary $\Sigma$ and suppose that a neighborhood of the boundary is isometric to $\Sigma \times [0,L)$ for some $L>0$. 
Let $\Omega$ be a domain in $M$ with smooth boundary that contains $\Sigma$ and suppose that $\Sigma \times [0,L/2] \subset \Omega^o$. 
Let $b$ be the number of connected components of $\Omega$ that have non-empty intersection with $\Sigma$. Then for $\epsilon>0$ there exists a conformal deformation of $g$, denoted $e^{\delta_{\epsilon}}g$, such that $\supp(\delta_{\epsilon}) \subset \Omega^o$ and:
\begin{align*}
\lim_{\epsilon \to 0} \sigma_{b+1} = + \infty.
\end{align*}
\end{proposition}

\begin{proof} First, we construct the conformal deformation of $g$. We will use the same notation as in Proposition \ref{half_theorem}. Let $\delta_{\epsilon}\geq 0$ be a smooth function that is identically $-2\log(\epsilon)$ on $(\Sigma \times [0,L/2))^c \cap \Omega$ and is zero on $\left(\cup_i \widetilde{\Omega}_i\right)^c$ and $\Sigma \times [0,\epsilon)$. Set $g_{\epsilon} = e^{2\delta_{\epsilon}}g$. 
Let $u$ be a smooth function on $M$ that is jointly orthogonal to constant functions on $\Sigma_i$, for $1 \leq i \leq b$. Notice that $u$ is orthogonal to $\myspan \{ \psi_{1},...,\psi_b \}$ on the boundary. Moreover, assume that $\int_{\Sigma} |u|^2 \dvol(g_\Sigma)=1$ and (possibly after relabeling the components of $\Omega$) that $\int_{\Sigma_1} |u|^2 \dvol(g_\Sigma) \geq \frac{1}{b}$. Consider the Steklov-Neumann eigenvalue problem on $\Omega_1$ with Steklov boundary conditions on $\Sigma_1$ and Neumann boundary conditions on $\partial_I \Omega_1$. Then we have:
\begin{align*}
\sigma_{b+1}(g_\epsilon) \geq \int_M |du|^2 \dvol(g_\epsilon) \geq \frac{\int_{\Omega_1} |du|^2 \dvol(g_\epsilon)}{b \int_{\Sigma_1} |u|^2 \dvol(g_\Sigma)} \geq \frac{1}{b}\sigma^N_{2}\left(\Omega_1,g_{\epsilon} \right),
\end{align*}
However, by Lemma \ref{mixed} and the construction of $g_\epsilon$, we know that there is a constant $A$ for which $\sigma^N_{2}\left(\Omega_1,g_\epsilon\right) \geq \frac{A}{\epsilon}$, which completes the proof.  
\end{proof}

Having produced a large gap in the Steklov spectrum using conformal deformations of $(M,g)$ when $g$ is a product metric in a neighborhood of $\Sigma$, we use a quasi-isometry argument to deduce Theorem \ref{thm:localizedgap} for a more general metric. Recall that two metrics $g_1$ and $g_2$ on $M$ are {\it quasi-isometric} with ratio $A \geq 1$ if for every $x \in M$ and $v \in T_xM$ we have:
\begin{align*}
\frac{1}{A} \leq \frac{g_1(x)(v,v)}{g_2(x)(v,v)} \leq A.
\end{align*}
Using the variational characterization of the Steklov eigenvalues, one can deduce that if $g_1$ and $g_2$ are quasi-isometric with ratio $A$ (and $\dim(M) = n$) then:
\begin{align*}
\frac{1}{A^{2n+1}} \leq \frac{\sigma_k(M,g_1)}{\sigma_k(M,g_2)} \leq A^{2n+1}.
\end{align*}
For a proof of this fact, see Proposition 2.2 of \cite{girouard_preprint}. 
\begin{proof}[Proof of Theorem \ref{thm:localizedgap}] Fix some domain $\Omega \subset M$ with $b$ connected components with non-empty exterior boundary on which we allow conformal deformations. That $\sigma^{\#}_{b}(\Omega) < \infty$ was the content of Proposition \ref{half_theorem}. Thus, it suffices to construct a conformal deformation in $\Omega$ such that $\sigma_{b+1}(g_\epsilon) \to \infty$. 
Given $\delta>0$, let
\begin{align*}
N(\delta) = \left \{ p \in M; \hspace{4 pt} d_g(p, \Sigma) <\delta \right \}. 
\end{align*}
The normal exponential map along the boundary defines Fermi coordinates on a neighborhood $V$ of the boundary. The distance $t$ to the boundary is one of the coordinates. It follows from Gauss' Lemma that the Riemannian metric is expressed as $g = h_t + dt^2$, where $h_t$ is the restriction of $g$ to the parallel hypersurface $\Sigma_t$ at distance $t$ from the boundary $\Sigma$. Moreover, $h_0 = g_\Sigma$ and, by continuity, there exists $\delta>0$ such that the restriction of $h_t$ to $\Sigma_t$ is quasi-isometric to $g_\Sigma$ with ratio two for $0 \leq t \leq 3 \delta$. Assume also that $\delta$ is sufficiently small so that $N(\delta/2) \subset \Omega^o$. Let $\chi \colon [0, 3 \delta] \to \mathbb{R}$ be a smooth non-decreasing function satisfying:
\begin{align*}
\chi(t) = \begin{cases}
0 \hspace{4 pt} \text{for } 0\leq t \leq \delta\\
1 \hspace{4 pt} \text{for }2\delta \leq t \leq 3\delta.
\end{cases}
\end{align*}
Using Fermi coordinates, define the following metric:
\begin{align*}
g_{\delta} = \begin{cases} \chi(t)(h_t+dt^2)+(1-\chi(t))(g_\Sigma + dt^2) \hspace{5 pt} \text{in }N(3 \delta)&\\
g \hspace{5 pt} \text{otherwise.}&
\end{cases}
\end{align*}
A few observations: by construction, $g_{\delta}$ is quasi-isometric to $g$ with ratio two, and it is isometric to a Riemannian product $g_\Sigma +dt^2$ on $N(\delta)$. We now apply Proposition \ref{auxiliary} to $(M, g_\delta)$. That is, using Proposition \ref{auxiliary}, we know that there is a family of conformal deformations $e^{2f_\epsilon}g_\delta$ such that: 
\begin{align*}
\sigma_{b+1}(M, e^{2f_\epsilon}g_\delta) \geq C/\epsilon.
\end{align*}
where $C>0$ is an explicit constant that depends on $g_\delta$ but not on $\epsilon$. Since $e^{2f_\epsilon}g_\delta$ is quasi-isometric to $e^{2f_\epsilon}g$ with the same constant of two, the theorem follows. 
\end{proof}

\bibliographystyle{abbrv}
\bibliography{mybib}

\end{document}